\documentclass[12pt]{article}

\usepackage[pdftex]{graphicx,color}
\graphicspath{{pic/}}
\DeclareGraphicsExtensions { . jpg , . png , . JPG}
\usepackage[centertags]{amsmath}
\usepackage[T1]{fontenc}
\usepackage{amsfonts}
\usepackage{amssymb}
\usepackage{amsthm}
\usepackage{newlfont}

\newlength{\defbaselineskip}
\newcommand{\setlinespacing}[1]%
           {\setlength{\baselineskip}{#1 \defbaselineskip}}

\theoremstyle{plain}
\newtheorem{thm}{Theorem}[section]
\newtheorem{cor}[thm]{Corollary}

\newtheorem{prop}[thm]{Proposition}
\theoremstyle{definition}
\newtheorem{defn}{Definition}[section]
\theoremstyle{remark}
\newtheorem{rem}{Remark}[section]
\numberwithin{equation}{section}

\title{The homogenous tree as an electric network} 
\author{Alice Vatamanelu\\
Dept. of Mathematics\\
Purdue University\\
vatamanelu@gmail.com} 
\date{\today} 

\begin{document}
\maketitle 
\abstract{Let T be an infinite homogenous tree of homogeneity $q+1$. Attaching to each edge the conductance $1$, the tree will became an electric network. The reversible Markov chain associated to this network is the simple random walk on the homogenous tree. Using results regarding the equivalence between a reversible Markov chain and an electric network, we will express voltages, currents, the Green fuction hitting times, transitions number, probabilities of reaching a set before another, as functions of the distance on the homogenous tree. This connection enables us to give simpler proofs for the properties of the random walk under discussion.}

\section{Introduction}   

It is well-known that every reversible Markov chain equivalent to an electric network (see \cite{DS}, \cite{MR0124932}). More precisely, let us consider a reversible Markov chain of state space $\mathbf{G}$ ($\exists \pi : \mathbf{G}\rightarrow ]0,\infty[,\quad \pi(x)p_{xy}=\pi(y)p_{yx}\quad \forall x,y$) and transition matrix $P=(P_{xy})_{x,y\in \mathbf{G}}$. We will associate to this chain an electric network (a weighted graph, the weights being thought as conductances) in the following way: we take the vertex set of the graph to be $\mathbf{G}$, we define $x\sim y$ if $p_{xy}> 0$ and take the weights to be $c(x,y)=\prod(x)p_{xy}$. So $(\mathbf{G},\sim,c)$ will become a weighted graph. Conversely, let's start from a weighted graph $(\mathbf{G},\sim,c)$. We will consider il to be countable and satisfying $\sum c(x,y)<\infty$.\\
The associated reversible Markov chain will be created as follows: we will take $\mathbf{G}$ as a state space, we will define its transition matrix to be $P=(p_{xy})_{x,y\in \mathbf{G}}$, $p_{xy}=\frac{c(x,y)}{\sum_{y\sim x}c(x,y)}$ for $y\sim x$ and $0$ otherwise. We will take the reversibility function $\pi$ to be $\pi(x)=\sum_{y\sim x}c(x,y)$. One can easily check that the matrix $P$ is stochastic and that the function $\pi$ satisfies the condition $\pi(x)p_{xy}=\pi(y)p_{yx}\quad \forall x,y$. So $(\mathbf{G},P,\pi)$ will be a reversible Markov chain.\\
Taking into account this equivalence, we will not make any distinction in what follows between the notions of reversible Markov chain and electric network.\\
Applying the procedure above to the network obtained by attaching unit conductances to the edges of the homogenous tree, we will obtain that its attached Markov chain is the one having the matrix $P=(p_{xy})_{x,y\in T}$, $p_{x,y}=\frac{1}{q+1}$ for $y\sim x$ and $0$ otherwise. This is exactly the simple random walk on the tree (For more details on random walks on trees we refer to \cite{MR1001523} and \cite{MR1743100}).\\
The paper will analyse this process viewing it as an electric network. We will obtain our main results using electric network techniques.

\

\textbf{Acknowledgments}: The author kindly thanks M. Abbassi and F. Baudoin for their kind help.

\  

\section{The probabilistic framework}
The probabilistic framework of this paper will be the following:\\
Let us consider an irreducible and reversible Markov chain, having the following elements:\\
\textbf{The state space:} $\mathbf{G}$\\
\textbf{The transition matrix:}
\[P=(p_{xy})_{x,y\in \mathbf{G}}\]
\textbf{The reversibility function:}
\[\pi:\mathbf{G}\rightarrow \mathbb{R}^{*}, \pi(x)p_{xy}=\pi(y)p_{xy}\quad \forall x,y\]
\textbf{The equivalent states:}
\[x\sim y\Leftrightarrow^{def}p_{xy}>0\]
\textbf{The hitting vector} of a set $M\subset \mathbf{G}$:
\[T_{M}=\inf\{n\geq 1: X_{n}\in M\}.\]
When we have $M=\{a\}$, the hitting time will be denoted by $\tau_{a}$.\\
Let us distinguish between $\tau_{M}$ and $T_{M}$, where $T_{M}$ is the following:
\[T_{M}=\inf\{n\geq 1: X_{n}\in M\}.\]
When we have $M=\{a\}$, $T_{M}$ will be denoted by $T_{a}$.\\
\textbf{The hitting vector} of a set $M\subset \mathbf{G}$:
\[\nu^{M}:=\left(P_{x}(\tau_{M}<\infty)\right)_{x\in\mathbf{G}}.\]
\textbf{The transience:}\\
The chain is transient $\Leftrightarrow^{def}P_{a}(T_{a}<\infty)<1\quad \forall a$\\
\textbf{The Green function:}
\[G(x,y):=\delta_{xy}+p_{xy}+p_{xy}^{2}+p_{xy}^{3}+...,\quad \forall x,y \in \mathbf{G},\]
Where $p_{xy}^{n}$ is the $xy$ element of matrix $p^{n}$. The Green function can also be expressed in the following way (see [S], theorem 10.1):
\[G(x,y)=E_{x}[\sum_{n\geq 0}1_{x_{n}=y}].\]
\textbf{The harmonic functions in a point x:}
\[f:\mathbf{G}\rightarrow \mathbb{R},\text{ having } Pf(x)=f(x),\]
\[\text{that is }\sum_{y\sim x}p_{xy}f(y)=f(x).\]
\textbf{The harmonic functions on a set} $\mathbf{M}\in \mathbf{G}$:\\
\[f:\mathbf{G}\rightarrow\mathbb{R},\text{ f being harmonic in every } x\in M.\]
\textbf{The transitions number} from $x$ to $y$ (after one step):\\
\[S_{xy}:=\sum_{n\geq 0}1_{\{x_{n}=x,x_{n+1}=y\}}, \text{ for } x\sim y.\]
\section{The electric network framework (first part)}
The first part of the electric network framework of this paper will be described by the following definition:
\begin{defn}
Let us consider an electric network having the following elements:\\
a) $\mathbf{G}$: \textbf{the vertex set.}\\
b) E: \textbf{the edges set} (E contains its edge with its both possible orientations).\\
c) v: \textbf{the voltage function} that appears as a consequence of hooking up a battery between two vertices of the network ($v: \mathbf{G}\rightarrow R$).\\
d) i: \textbf{the electric current intensity function} that appears as a cosequence of hooking up a battery between two vertices of the network ($i: E\rightarrow R$).\\
e) c: \textbf{the conductance function} defined on the set of the oriented edges of the graph ($c: E\rightarrow R$); c is a symmetric function.\\
f) r: \textbf{the resistance function} defined on the set of the oriented edges of the graph ($r: E\rightarrow R$); r is the inverse of the function c, so it is symmetric.\\
g) $C(a\leftrightarrow Z)$ \textbf{the effective conductance from} $\mathbf{a}$ \textbf{to} $\mathbf{Z}$ where $\{a\}$ and $Z$ are two disjoint subsets of $\mathbf{G}$, this notion is defined for a finite graph; it will be the ratio between the sum of the currents that enter the current through the point $a$ and the strictly positive voltage in $a$ when hooking up a battery between $a$ and $Z$ with $v/Z=a$:
\[C(a\leftrightarrow Z)=\frac{\sum_{x\sim a}i(a,x)}{v(a)}.\]
h) $R(a\leftrightarrow Z)$: \textbf{the efective resistance from} $\mathbf{a}$ \textbf{to} $\mathbf{Z}$, where $\{a\}$ and $Z$ are two disjoint subsets of $\mathbf{G}$, this notion is defined for a finite graph, to be the inverse of the effective conductance from $a$ to $Z$.\\
i) $C (a\leftrightarrow Z)$: \textbf{the effective conductance from} $\mathbf{a}$ \textbf{to infinity}, this notion is defined for an infinite graph G; for defining it, we have to exhaust the graph $\mathbf{G}$ by a sequence $(\mathbf{G}_{n})_{n\geq 0}$ of finite subgraphs. This can be done, for instance, as follows:
\[\mathbf{G}_{0}:=\{0\} \text{ (an arbitrarily fixed root)}\]
\[\mathbf{G}_{n}:=\mathbf{G}_{n-1}\cup\{y:\exists x\in\mathbf{G}, y\sim x\},\quad \forall n\geq 1.\]
One can easily check that $\mathbf{G}_{n-1}\leq \mathbf{G}_{n}\quad \forall n\geq 1$ and that $\cup_{n}\mathbf{G}_{n}=\mathbf{G}$.\\
Let us now denote $\mathbf{G}\backslash\mathbf{G}_{n}$ by $Z_{n}$. For each $n\geq 1$, we identify all the vertices of $C(a\leftrightarrow z_{n})$ the effective conductance from $a$ to $z_{n}$ in $\mathbf{G}^{(n)}$.\\
Finally, we define $C(a\leftrightarrow \infty)$ as follows:
\[C(a\leftrightarrow \infty):=\lim_{n}C(a\leftrightarrow z_{n}).\]
j) $R(a\leftrightarrow \infty)$: \textbf{the effective resistance from} $\mathbf{a}$ \textbf{to infinity.}\\
We define this notion for an infinite graph; it will be the inverse of the effective conductance from $a$ to infinity.
\end{defn}
In the above framework, it will be useful for the purpose of this paper to present the following characterisation of transience by means of conductances (see [LP], theorem 2.3).
\begin{thm}
An infinite electric network $\mathbf{G}$ is \textbf{transient iff the effective conductance from any vertex to infinity is strictly positive.}
\end{thm}
The idea of proof is the following (for details see \cite{LP}):\\
We start from the following proposition, called the Maximum Principle:\\
THE MAXIMUM PRINCIPLE: Let $\mathbf{G}$ be a finite graph, $\mathbf{H}\subseteq \mathbf{G}$, $\mathbf{H}$ connected and $\overline{\mathbf{H}}:=\{y\in \mathbf{G}:\exists x\in \mathbf{H}\text{ s.t. } y\sim x\}$. Let $f:\mathbf{G}\rightarrow \mathbb{R}$ $f$ harmonic on $\mathbf{H}$ and having $\max_{\mathbf{H}}f=\max_{\mathbf{G}}f$. Then $f/\overline{\mathbf{H}}=\max_{\mathbf{G}}f$. (If $f$ attains its maximum on a set where it is harmonic, then $f$ will be constant on that set).\\
By means of this Maximum Principle, we prove the so - called Uniqueness Principle:\\
THE UNIQUENESS PRINCIPLE: Let $\mathbf{G}$ be a finite graph, $\mathbf{A}\subset\mathbf{G}$ $f$ and $g$ two real functions on $\mathbf{G}$, harmonic on $\mathbf{H}$ and such that $f=g$ on $\mathbf{H}$. Then $f=g$. (A function is perfectly determined by its harmonicity on a certain set and by its values on the complementary of this set).\\
This principle will be used to prove the coincidence of the following two real functions defined on a finite graph $\mathbf{G}$:\\
$\cdot$ the function $F(x)=P_{x}(\tau_{a}<\tau_{z})$, where $\{a\}$ and $Z$ are fixed, disjoint subsets of $\mathbf{G}$.\\
$\cdot$ the voltage function $v$ that will appear on $\mathbf{G}$ as a consequence of hooking up a battery between $a$ and $Z$ with $v(a)=1$ and $v/Z=0$.\\
These functions are both harmonic on $\left(\{a\}\cup Z\right)$ and they take the same values on $\{a\}\cup Z$, so, according to the Uniqueness Principle, they will coincide. Thus we obtain the following probabilistic expression of the voltage function:
\[v(x)=P_{x}(\tau_{a}<\tau_{z}).\]
This fact will allow us to express the effective conductance from $a$ to $Z$ in a probabilistic way:
\[C(a\leftrightarrow Z)=\pi (a)P[a\rightarrow Z]\]
The chain is transient iff $P_{a}(T_{a}=\infty)>0 \quad \forall a$, that is iff $\lim_{n}P[a\rightarrow Z_{n}]>0$, where $Z_{n}$ are those in defenition 3.1 i). Now, using (3.1), we will obtain the required characterisation of transience.
\section{The electric network framwork (second part)}
The second part of the electric network framework of this paper will be describer through the following three definitions:
\begin{defn}
We consider an (electric) network having the following elements:\\
a)$\mathbf{G}$: \textbf{the vertex set.}\\
b)E: \textbf{the edges set} (E contains each edge with its both possible orientations).\\
c)$\breve{e}$: \textbf{the opposite} of the edge $e-e^{-}e^{+}$; $e^{+}$ is the \textbf{head} of $e$.\\
d)$l^{2}(\mathbf{}G)$: the Hilbert space of the \textbf{real functions on} $\mathbf{G}$ with the property $\sum_{x\in \mathbf{G}}f^{2}(x)<\infty$, with the scaler product
\[(f,g)=\sum_{x\in \mathbf{G}}f(x)g(x).\]
e)$l^{2}-(E)$: the Hilbert space of the \textbf{real antisymetric functions on E} with the property $\sum_{e\in E}\theta^{2}(e)<\infty$, with the scalar product
\[(\theta,\theta^{'})=\frac{1}{2}\sum_{e\in E}\theta (e)\theta^{'}(e).\]
f) the operator $d$:
\[d: \quad l^{2}(\mathbf{G})\rightarrow l^{2}-(E),\quad dF(e)=F(e^{-})-F(e^{+}).\]
g) the operator $d^{*}$:
\[d^{*}:\quad l^{2}-(E)\rightarrow l^{2}(\mathbf{G}), \quad d^{*}\theta (x)=\sum_{e^{-}=x}\theta (e).\]
\end{defn}
\begin{defn}
For this definition, we will assume that every function $\theta \in l^{2}-(E)$ represents a certain type of liquid that flows through the network. Let $\mathbf{G}$ be a finite network for the points a),...,f) and an infinite one for the point g).\\
a) \textbf{The quantity of liquid} (of type $\theta$) \textbf{that enters into the network through a vertex} $\mathbf{a}$ is by definition $d^{*}\theta (a)$.\\
b) For $A$ and $Z$ two disjoint, fixed subset of $\mathbf{G}$ (intuitively thought as a source and, respectively, \textbf{exit point} of the flow), we will call a function $\theta \in l^{2}-(E)$ a \textbf{flow from A to Z} if:
\[d^{*}\theta > 0 \textbf{ on } A\]
\[d^{*}\theta < 0 \textbf{ on } Z\]
\[d^{*}\theta = 0 \textbf{ on } (A\cup Z).\]
c) If $\theta$ is a flow from $A$ to $Z$. \textbf{the quantity of liquid that enters into the network} is by definition $\sum_{a \in A}d^{*}\theta (a)$.\\
d) If $\theta$ is a flow from $A$ to $Z$. \textbf{the quantity of liquid that flows out of the network} is by definition $\sum_{z \in Z}d^{*}\theta (z)$.\\
e) If $\theta$ is a flow from $A$ to $Z$, we define $\mathbf{strength(\theta)}$ to be $\sum_{a \in A}d^{*}\theta (a)$.\\
f) A \textbf{unit flow} is by definition a flow of strength $1$.\\
g) An antisymmetric function $\theta: \quad E\rightarrow\mathbb{R}$ is called a \textbf{unit flow from} $\mathbf{a}$ \textbf{to infinity} if $d^{*}\theta=1_{\{a\}}$.\\
\end{defn}
\begin{defn}
a) On the space $l^{2}-(E)$ we define the scalar product $(.;.)_{r}$ in the following way:
\[(\theta,\theta^{'})_{r}=\frac{1}{2}\sum_{e\in E}\theta (e)\theta^{'}(e)r(e).\]
b) \textbf{The energy} of an antisymmetric function $\theta: \quad E\rightarrow \mathbb{R}$ is defined to be
\[\varepsilon (\theta)=\parallel\theta\parallel_{r}^{2}.\]
c)We define \textbf{the unit flow along} $\mathbf{e}$ to be the function $\chi^{e}\in l^{2}-(E)$,
\[\chi^{e}:=1_{e}-1_{\breve{e}}.\]
d)We define the spaces
\[\begin{array}{c}
     \bigstar = sp\{\sum_{e^{-}=x}c(e)\chi^{e}; x\in \mathbf{G}\}\quad (\textbf{star space})\\
    \lozenge=sp\{\sum_{i=1}^{n}\chi^{e_{i}}; e_{1},...,e_{n}\in E \text{ oriented cycle; } n\geq 0\}\quad (\textbf{cycle space})
  \end{array}
\]
\end{defn}

In the above framework, it will be useful for the purpose of this paper to present the following characterisation of transience by means of flows (see [LP], theorem 2.10).
\begin{thm}
An infinite electric network $\mathbf{G}$ is \textbf{transient iff there is a unit flow on} $\mathbf{G}$ \textbf{of finite energy from any vertex to infinity.}
\end{thm}
the steps of the proof are the following (for details see \cite{LP}):\\
1) Proving that the operators $d$ and $d^{*}$ from definition 4.1 f), g) are adjoint in the sense that
\[(\theta,dF)=(d^{*}\theta,F),\]
for any $\theta \in l^{2}-(E)$ and any $F\in l^{2}(\mathbf{G})$.\\
2) Writing down Ohm and Kirchhoff Laws by means of the operators $d$ and $d^{*}$:
\[\text{Ohm's Law: } dv=i.r\]
\[\text{Kirchhoff's Law: } d^{*}i(x)=0,\]
if $x$ is not connected to any battery.\\
3) Proving the two properties of flows:\\
a) $\sum_{a\in A}d^{*}\theta (a)=-\sum_{a\in Z}d^{*}\theta (a)$\\
(The quantity of liquid that enters into the network is equal to the quantity of liquid that flows out of the network).\\
b) $(\theta,dF)=Strength(\theta)[F(A)-F(Z)],$\\
for any real function $F$ that is constant on $A$ and $Z$.\\
4) Writing down Kirchhoff's Law by mean's of the scalar product $(.;.)_{r}$:
\[\text{The node's Law: }\left(\sum_{e^{-}=x}c(e)\chi^{e},i\right)_{r}=0,\]
if $x$ is not connected to any battery.
\[\text{The cycles's Law: }\left(\sum_{j=1}^{n}\chi^{ej},i\right)_{r}=0,\]
for any oriented cycle $e_{1},...,e_{n}$.\\
5) Proving that the space $l^{2}-(E)$ is the direct sum of the orthogonal subspaces $\bigstar$ and $\lozenge$.\\
6) Proving Thomson's Principle:\\
THOMSON'S PRINCIPLE: Let $\mathbf{G}$ be a finite network, $A$ and $Z$ two disjoint subsets of $\mathbf{G}$, $\theta$ a unit flow from $A$ to $Z$ and $i$ the electric current flow from $A$ to $Z$ with $d^{*}i=d^{*}\theta$. Then
\[\varepsilon (\theta)\geq \varepsilon (i).\]
(Of all the flow from $A$ to $Z$ having the same $d^{*}$, the electric current flow is the energy minimizes).\\
7)Completing the proof of the theorem by means of the previous six steps.

\section{The basic theorem}
The two characterisations of the transience of an infinite network presented in sections $3$ and $4$ are the background for the following theorem (see \cite{LP}, proposition 2.11), whose importance for our future purposes the following facts:\\
1) It extends Ohm's law from the finite to the infinite case.\\
2) It gives the electric expression of the hitting vector for an infinite network.\\
3) It gives the electric expression of the Green function for an infinite network.
\begin{thm}
Let $\mathbf{G}$ be a trasient network and $(\mathbf{G}_{n})_{n\geq 0}$ a sequence of finite subgraphs, containing a vertex a, that exhaust $\mathbf{G}$. We contract to $z_{n}$ the vertices outside $\mathbf{G}$, forming $\mathbf{G}^{(n)}$ (see Figure \ref{fig1}).\\
\end{thm}
Let $i_{n}$ be the unit current flow $i_{n} C^{n}$ from $a$ to $z_{n}$. Then $(i_{n})_{n}$ has a point wise limit $i$ on $\mathbf{G}$, that is the unique unit flow on $\mathbf{G}$ from $a$ to infinity of minimum energy.\\
Let $v_{n}$ be the voltages on $\mathbf{G}^{(n)}$ corresponding to $i_{n}$ and with $v_{n}(z_{n})=0$. Then $v:=\lim_{n}v_{n}$ exists on $\mathbf{G}$, is finite and has the properties:\\
1) $dv=ir$\\
2) $v(a)=\varepsilon (i)=R(a\leftrightarrow \infty)$\\
3) $\frac{v(x)}{v(a)}=P_{x}[\tau_{a}<\infty],\quad \forall x$.\\
Let's start in a the random walk on $\mathbf{G}$. For any $x$, the expected number of visits to $x$ is
\[\mathbf{G}(a,x)=\pi (x)v(x).\]
For any edge $e$, the expected signed number of crossings of $e$ is $i(e)$.\\
For the proof see [LP].
\begin{figure}
\includegraphics{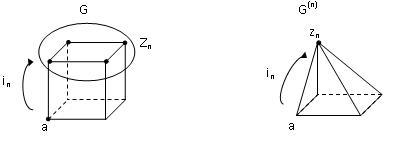}\\
\caption{}
\label{fig1}
\end{figure}

\section{The homogenous tree as an electric network}
\begin{defn}
a) A \textbf{tree} is a graph that is locally finite, connected and without loops. Notation: for $x$ and $y$ neighborn, we will write $x\sim y$.\\
b) A \textbf{path} is a finite or infinie sequence of vertices $[v_{0},v_{1},...]$ such that $v_{k}\sim v_{k+1}$ for any $k\geq 0$.\\
c) A \textbf{geodesic path} is a path $[v_{0},v_{1},...]$ such that $v_{k-1}\sim v_{k+1}$ for any $k\geq 0$.
\end{defn}
The tree we will consider will be \textbf{infinite} (most of the time), homogenous, of degree $q+1$ (each vertex has exactly $q+1$neighbours), $q \geq 2$. It can be viewed as in figures $\ref{fig2}$ or $\ref{fig3}$. We will denote by $0$ an arbitrary vertex, fixed as a root. We will denote by $T$ the vertex set.
\begin{figure}
\includegraphics{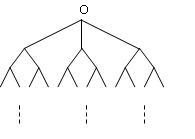}\\
\caption{}
\label{fig2}
\end{figure}
\begin{figure}
\includegraphics{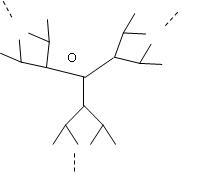}\\
\caption{}
\label{fig3}
\end{figure}

\begin{defn}
a) We define the distance $\mathbf{d(u,v)}$ between two vertices $u$ and $v$ to be the number of edges of the geodesic path from $u$ to $v$.\\
b)\textbf{The length of a vertex v} is by definition $d(0,v)$; it will be denoted by $|v|$.
\end{defn}
We will attach to each edge the conductance $1$. So the tree will become an electric network. The relation $\sum_{y \sim x}c(x,y)<\infty$ will be true for any vertex $x$, then, as in the introduction, we will associate a reversible Markov chain to this network. This is the chain of matrix $P=(p_{xy})_{x,y\in T}$,
\[p_{xy}=\begin{array}{c}
           \frac{1}{q+1},\text{ if } y\sim x \\
           0,\text{ otherwise}
         \end{array}
\]

It represents the simple random walk on the homogenous tree. Writing the relation $c(x,y)=\pi (x)p_{xy}$ for an arbitrary edge $xy$, we obtain the expression of the reversibility function $\pi$:
\[\pi:\quad T\rightarrow \mathbb{R}_{+}^{*}, \quad\quad \pi  (x)=q+1.\]
\begin{rem}
This tree has finite resistance to infinity it represents the prototype of network that has finite resistance to infinity.
\end{rem}
\begin{proof}
Let's compute the effective resistance from $0\text{ to } \infty$. According to definition $3.1$ j), we will have:\\
\[R(0\leftrightarrow \infty)=\lim_{n}R(0\leftrightarrow z_{n}),\]
where $z_{n}$ is the point which concentrates all the points of level $\geq n$.\\
During this procedure, we throw away the resulting loops, so it is just like concentrating in $z_{n}$ all the vertices of level $n$ (see figure \ref{fig4}).
\begin{figure}
\includegraphics{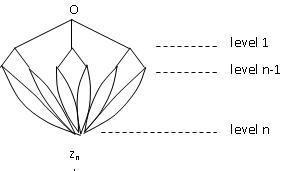}\\
\caption{}
\label{fig4}
\end{figure}
We have to transform this network successively into a simpler one, in order to compute $R(0\leftrightarrow z_{n})$ (see figure \ref{fig5})
\begin{figure}
\includegraphics{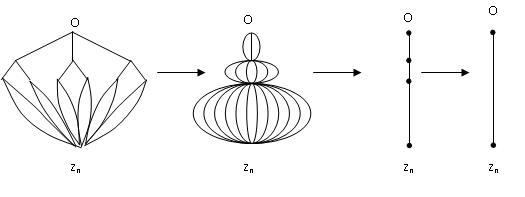}\\
\caption{}
\label{fig5}
\end{figure}
We have:\\
a) Unit conductances.\\
b) Unit conductances.\\
c) The conductances are respectively equal to $q+1$, $q(q+1)$, $q^2 (q+1)$,...\\
d) The resistance is equal to $\frac{1}{q+1}+\frac{1}{q(q+1)}+\frac{1}{q^{2}(q+1)}+...$\\
So,
\[R(0\leftrightarrow z_{n})=\frac{1}{q+1}.\frac{1-\left(\frac{1}{q}\right)^{2}}{1-\frac{1}{q}}\]
Consequently,
\[R(0\leftrightarrow z_{n})=\lim_{n}\frac{1}{q+1}.\frac{1-\left(\frac{1}{q}\right)^{2}}{1-\frac{1}{q}}=\frac{q}{q^2-1}<\infty.\]
\end{proof}
\begin{rem}
Having a finite resistance from any vertex to infinity, the homogenous tree will have a strictly positive conductance from any vertex to infinity. Thus, according to theorem 3.1, the associated Markov chain will be transient.
\end{rem}
\section{New results}
We present now the final purpose of this paper: results upon the electric network of the homogenous tree.We first prove a proposition that will express voltages and currents intensities as functions of distance to the root.
\begin{prop}
Let $T$ be the transient network given by the infinite homogenous tree of degree q+1. We take root $0$ to be the reference point in theorem 5.1. Then:\\
a) The limit function $i$ given by the same theorem is:
\[i(x,y)=\begin{array}{c}
           \frac{1}{q+1}.\frac{1}{q^{\mid x \mid}}, \quad \forall x,y\in T, \quad y\sim x, \mid y\mid > \mid x\mid\\
           -\frac{1}{q+1}.\frac{1}{q^{\mid x\mid -1}}, \quad x,y \in T, \quad y\sim x, \mid y\mid <\mid x\mid
         \end{array}
\]
The limit function $v$ given by the same theorem is:
\[v(x)=\frac{q}{q^2-1}.q^{-\mid x\mid},\quad \forall x\in T.\]
\end{prop}
\begin{proof}
a) Let $(T_{n})_{n\geq 0}$ be the following sequence of finite subtrees, containing $0$, that exhaust $T$:
\[T_{0},\quad T_{n}=\{x\in T, \mid x\mid \leq n\},\quad \forall n\geq 1.\]
We contract to $z_{n}$ the vertices outside $T_{n}$ and denote $T^{(n)}$ the union between $T_{n}$ and $\{z_{n}\}$. (see \ref{fig6}). Let in be the unit current flow in $T^{(n)}$ from $0$ to $z_{n}$. Let $v_{n}$ be the voltage on $T^{(n)}$ corresponding to $i_{n}$ and with $v_{n}(z_{n})=0$. Let $T_{1}^{(n)},...,T_{q+1}^{(n)}$ be the $q+1$ circuits that appear between $0$ and $z_{n}$ when we hook up a battery between $0$ and $z_{n}$, with $v_{n}(0)>0$ and $v_{n}(z_{n})=0$ (see figure \ref{fig7}).\\
The homogenity of the edges resistors will imply:
\[R^{T_{1}^{n}}(0\leftrightarrow z_{n})=...=R^{T_{q+1}^{n}}(0\leftrightarrow z_{n}).\]
Hence, according to Ohm's Law, written $q+1$ times between $0$ and $z_{n}$,
\[i_{n}(0,x_{1})=...=i_{n}(0,x_{q+1}),\]
$x_{1},...,x_{q+1}$ being the points of level 1.\\
We know $d^{*}i_{n}=1$, so we will have:
\[i_{n}(0,x_{1})=...=i_{n}(0,x_{q+1})=\frac{1}{q+1}.\]
Consequently, Ohm's Law written respectively for $0$ and $x_{1},...,0$ and $x_{q+1}$ will imply:
\[v_{n}(x_{1})=...=v_{n}(x_{q+1})\]
The above procedure for the points $0$ and $z_{n}$ will be now applied, successively, to the points $x_{1}$ and $z_{n},...,x_{q+1}$ and $z_{n}$. We will obtain the intensities of the currents passing between the first two levels:
\[i_{n}(x_{1},...,y_{1})=...=i_{n}(x_{q+1},...,y_{q(q+1)})=\frac{1}{q+1}.\frac{1}{q},\]
where we have denoted by $y_{1},...,y_{q(q+1)}$ the points of the second level.\\
We obtain in a similar manner the intensities of the currents going between any two consecutive levels:
\[i_{n}(x,y)=\frac{1}{q+1}.\frac{1}{q^{\mid x\mid}}, \quad \forall x,y, x\sim y, \mid y\mid=\mid x\mid + 1.\]
Let us now consider arbitrary $x$ and $y$, $x\sim y$, $\mid y\mid=\mid x\mid +1$ $(T_{n})_{n}$ exhaust $T$, sowe can find $n_{0}$ such that $x,y\in T_{n_{0}}$. Applying the above procedure to each $T_{n}$, $n\geq n_{0}$, we get:
\[i_{n}(x,y)=\frac{1}{q+1}.\frac{1}{q^{\mid x\mid}}, \quad \forall n\geq n_{0},\]
hence, by passing to the limit,
\[i_{n}(x,y)=\frac{1}{q+1}.\frac{1}{q^{\mid x\mid}}.\]
We will also have:
\[i_{n}(x,y)=-i_{n}(y,x)=\frac{1}{q+1}.\frac{1}{q^{\mid x\mid}}=-\frac{1}{q+1}.\frac{1}{q^{\mid y\mid-1}}\]
so we have proved a).
b) The proof will be completed by induction after $\mid x\mid$.\\
I. $\mid x\mid=1$: Ohm's Law between $0$ and $x$ will imply
\[v(0)-v(x)=i(0,x).1\]
But $v(0)=R(0\leftrightarrow \infty)=\frac{q}{q^{2}-1}$ according to theorem 5.1 and $i(0,x)=\frac{1}{q+1}$ according to a), so we get
\[\frac{q}{q^{2}-1}-v(x)=\frac{1}{q+1}.\]
Hence
\[v(x)=\frac{q}{q^2-1}.\frac{1}{q}\]
II. $\mid x\mid \sim \mid y \mid +1$
Ohm's law written between $x$ and $y$, $x\sim y$, $\mid y\mid=\mid x\mid + 1$, will imply:
\[v(x)-v(y)=i(x,y).1\]
But we have $v(x)=\frac{q}{q^{2}-1}.\frac{1}{q^\mid x\mid}$ and $i(x,y)=\frac{1}{q+1}.\frac{1}{q^{x}}$  according to a), so
\[v(y)=\frac{q}{q^{2}-1}.\frac{1}{q^{\mid x\mid}}-\frac{1}{q+1}.\frac{1}{q^{\mid x\mid}}=\frac{q}{q^{2}-1}.\frac{1}{q^{\mid x\mid +1}}.\]
I and II will imply
\[v(x)=\frac{q}{q^2-1}.\frac{1}{q^{\mid x\mid}}\quad \forall x\]
\end{proof}
\begin{figure}
\includegraphics{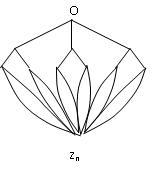}\\
\caption{}
\label{fig6}
\end{figure}
\begin{figure}
\includegraphics{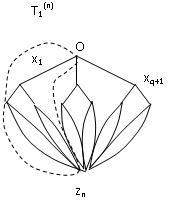}\\
\caption{}
\label{fig7}
\end{figure}
The corollary to come will present the particular form of the Green function for the transient network of the homogenous tree. Thus we arrive to the particular form of this function mentioned in \cite{GS}, Section 2.
\begin{cor}
Let $T$ be the transient network of the infinite homogenous tree off degree $q+1$ and $\mathbf{G}(.;.)$ the associated Green function. Then $\mathbf{G}(.;.)$ has the following expression:
\[\mathbf{G}(a,x)=\frac{q}{q-1}.q^{-d(a,x)},\quad \forall a,x\in T.\]
\end{cor}
\begin{proof}
Theorem 5.1 asserts that
\[\mathbf{G}(0,x)=\pi(x)v(x),\]
so according to proposition 7.1 b), we will get
\[\mathbf{G}(0,x)=(q+1).\frac{q}{q^{2}-1}.\frac{1}{q^{\mid x\mid}}=\frac{q}{q-1}.q^{-\mid x\mid}.\]
Writing this in the form
\[G(0,x)=\frac{q}{q-1}.q^{-d(0,x)}\]
and taking into account the fact that $0$ may be arbitrarily chosen, we get:
\[G(a,x)=\frac{q}{q-1}.q^{-d(a,x)}\quad \forall a,x\in T.\]
\end{proof}
The following proposition will express the hitting times and the transitions number on the homogenous tree network as functions of the distance to the root.
\begin{prop}
Let $T$ be the transient network of the infinite homogenous tree of degree $q+1$. Then\\
a) The hitting vector of the root is
\[\left(q^{-\mid x\mid}\right)_{x\in T}\]
b) The hitting vector of an arbitrary point $a\in T$ is
\[\left(q^{-d(a,x)}\right)_{x\in T}.\]
c) For any starting point $a$, the expected number of transitions of an edge $xy$ is
\[E_{a}[S_{xy}]=\frac{q}{q^{2}-1}.q^{-d(a,x)}.\]
\end{prop}
\begin{proof}
a) According to theorem 5.1, the hitting of the root is
\[Px(\tau_{a}<\infty)=\frac{v(x)}{v(0)}.\]
Using the voltages expressions given by proposition 7.1, we obtain
\[Px(\tau_{0}<\infty)=\frac{\frac{q}{q^{2}-1}.q^{-\mid x\mid}}{\frac{q}{q^{2}-1}}=q^{-\mid x\mid}.\]
b) Using the similarity property of the vertices (we have no intrinsec way to distinguish a vertex from another) and noticing that $\mid x\mid=d(0,x)$, we generalise a) writing that
\[\forall a, \forall x : \quad P_{x}(\tau_{a}<\infty)=q^{-d(a,x)}.\]
c) We have:
\[\forall a,\quad E_{a}[S_{xy}]=E_{a}[\sum_{n\geq 0}1_{\{x_{n}=x,x_{n+1}=y\}}]=E_{a}[\sum_{n\geq 0}1_{x_{n}=a}]p_{xy}=\mathbf{G}(a,x)p_{xy}\]
Using the particular form of the Green function for the homogenous tree and the relation $p_{xy}=\frac{1}{q+1}$, we get:
\[E_{a}[S_{xy}]=\frac{q}{q-1}.q^{-d(a,x)}.\frac{1}{q+1}=\frac{q}{q^{2}-1}/q^{-d(a,x)}.\]
\end{proof}

\begin{rem}
This expected number of transitions is exactly the voltage that appears in $x$ when the unit injection of current is made in $a$.
\end{rem}
Let $\{a\}$ and $Z$ be two disjoint subsets of the finite homogenous tree $T$. The last proposition will express the probability of reaching $Z$ before $a$ as a function of the degree and of the distance. The finite homogenous tree is obtained from the infinite one by keeping only the first $n$ levels.The edges conductances are still equal to one $P^{'}$, the transition matrix of the associated Markov chain, will obviously be finite. It can be expressed in the following way as function of the matrix $P$ of the infinite network:
\[p^{'}_{xy}=\begin{array}{c}
               1,\text{ for } x \text{ on level } n \text{ and } y\sim x \\
               p_{xy}, \text{ otherwise}.
             \end{array}
\]
On this finite tree, the effective resistance between two points $a$ and $x$ is exactly the distance between them. We can see this as follows: we hook up a battery between $a$ and $x$ and write Kirchhoff's Law for all vertices starting with the level $n$ ones and going up to those on the geodesic path between $a$ and $x$. The result will be that no current goes through any edge expert those on this geodesic path. So we can remove the edges outside it reducing the circuit to the resistors of the geodesic path from $a$ to $x$. The effective resistance from $a$ to $x$ will be the sum of their resistances. This is exactely $d(a,x)$.
\begin{prop}
Let $T$ be the transient network of the finite homogenous tree of degree $q+1$ and $0$ its root. We will have:
a) When $Z$ equals the level $n$ vertex set,
\[P[0\rightarrow Z]=q^{n-1}.\frac{q-1}{q^{n}-1}. \text{ see \ref{fig8}}\]
b) When the levels number is $n$, $a$ i an arbitrary level $n$ point and $Z$ equals the union of the subtrees rooted in $0$ and not containing $a$,
\[P[a\rightarrow Z]=\frac{1}{n}. \text{ see \ref{fig9}}\]
c) When the levels number is greater than $n$, $a$ is an arbitrary level $n$ point and $Z$ is the one in b),
\[P[a\rightarrow Z]=\frac{1}{n(q+1)}. (\text{see figure \ref{fig10}})\]
d)When $a$ is a terminal vertex and $x$ is arbitrary,
\[P[a\rightarrow x]=\frac{1}{d(a,x)}. (\text{see figure \ref{fig11}})\]
e)When $a$ is not terminal and $x$ is arbitrary,
\[P[a\rightarrow x]=\frac{1}{d(a,x).(q+1)}. (\text{see figure \ref{fig12}})\]
\begin{figure}
\includegraphics{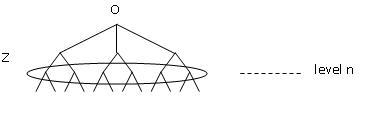}\\
\caption{}
\label{fig8}
\end{figure}
\begin{figure}
\includegraphics{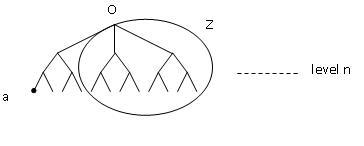}\\
\caption{}
\label{fig9}
\end{figure}
\begin{figure}
\includegraphics{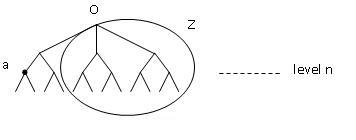}\\
\caption{}
\label{fig10}
\end{figure}
\begin{figure}
\includegraphics{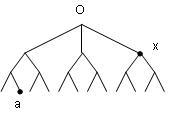}\\
\caption{}
\label{fig11}
\end{figure}
\begin{figure}
\includegraphics{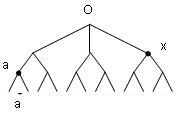}\\
\caption{}
\label{fig12}
\end{figure}
\end{prop}

\begin{proof}
We will have:\\
a) According to 3.1,
\[P[0\rightarrow Z]=\frac{C(0\leftrightarrow Z)}{\pi (0)}=\frac{1}{R(0\leftrightarrow Z)\pi (0)}\]
$R(0\leftrightarrow Z)$ is given by remark 6.1:
\[R(0\leftrightarrow Z)=\frac{1}{q+1}.\frac{1-\left(\frac{1}{q}\right)^{n}}{1-\frac{1}{q}}=\frac{1}{q^{2}-1}.\frac{q^{n}-1}{q^{n-1}}\]
$\pi (0)=q+1$, so we get:
\[R(0\leftrightarrow Z)=(q^{2}-1).\frac{q^{n-1}}{q^{n}-1}.\frac{1}{q+1}=q^{n-1}.\frac{q-1}{q^{n}-1}.\]
b) $P[a\rightarrow Z]=P[a\rightarrow 0]=\frac{C(a\leftrightarrow 0)}{\pi (a)}=\frac{1}{R(a\leftrightarrow 0)\pi (a)}=\frac{1}{n.1}=\frac{1}{n}$\\
c)$P[a\rightarrow Z]=\frac{1}{R(a\leftrightarrow 0)\pi (a)}=\frac{1}{n(q+1)}$\\
d) $P[a\rightarrow x]=\frac{C(a\leftrightarrow x)}{\pi (a)}=\frac{1}{R(a\leftrightarrow x)}.\frac{1}{\pi (a)}=\frac{1}{d(a,x).1}=\frac{1}{d(a,x)}$\\
e) $P[a\rightarrow x]=\frac{1}{R(a\leftrightarrow x)\pi (a)}=\frac{1}{d(a,x)(q+1)}$
\end{proof}
\bibliographystyle{amsplain}
\bibliography{docbib}

\providecommand{\bysame}{\leavevmode\hbox to3em{\hrulefill}\thinspace}
\providecommand{\MR}{\relax\ifhmode\unskip\space\fi MR }
\providecommand{\MRhref}[2]{%
  \href{http://www.ams.org/mathscinet-getitem?mr=#1}{#2}
}
\providecommand{\href}[2]{#2}
\begin{thebibliography}{1}

\bibitem{DS}
P.G. DOYLE and J.L. SNELL, \emph{Random walks and electric networks}, arxiv.
  math. PR/00010 57 \textbf{1} (2000).

\bibitem{GS}
K.~GOWRISANKARAN and D.~SINGMAN, \emph{Tangential limits of potentials on
  homogenous trees}, Potential Analysis \textbf{18} (2003), 77--96.

\bibitem{LP}
R.~LYON and Y.~PERES, \emph{Probability on trees and networks}, To be published
  by Cambridge University Press.

\bibitem{MR0124932}
C.~St. J.~A. Nash-Williams, \emph{Random walk and electric currents in
  networks}, Proc. Cambridge Philos. Soc. \textbf{55} (1959), 181--194.
  \MR{MR0124932 (23 \#A2239)}

\bibitem{MR1001523}
A.~Telcs, \emph{Random walks on graphs, electric networks and fractals},
  Probab. Theory Related Fields \textbf{82} (1989), no.~3, 435--449.
  \MR{MR1001523 (90h:60065)}

\bibitem{MR1743100}
Wolfgang Woess, \emph{Random walks on infinite graphs and groups}, Cambridge
  Tracts in Mathematics, vol. 138, Cambridge University Press, Cambridge, 2000.
  \MR{MR1743100 (2001k:60006)}

\end{thebibliography}


\end{document}